\documentclass[reqno,a4paper, 11pt]{amsart}

\usepackage[a4paper=true,pdfpagelabels]{hyperref}
\usepackage{graphicx}

\usepackage[ansinew]{inputenc}
\usepackage{amsfonts,epsfig}
\usepackage{latexsym}
\usepackage{amsmath}
\usepackage{amssymb}
\usepackage{mathabx}
\usepackage{comment}
\usepackage{color}

\newtheorem{theorem}{Theorem}
\newtheorem{lemma}[theorem]{Lemma}
\newtheorem{corollary}[theorem]{Corollary}

\newtheorem{lettertheorem}{Theorem}

\theoremstyle{definition}

\newtheorem{example}[theorem]{Example}

\theoremstyle{remark}

\numberwithin{equation}{section}

\newcommand{\D}{\mathbb{D}}

\newcommand{\N}{\mathbb{N}}

\newcommand{\HO}{\mathcal{H}}

\newcommand{\C}{\mathbb{C}}

\newcommand{\e}{\varepsilon}

\newcommand{\T}{\mathbb{T}}
\newcommand{\Th}{\Theta}

\def\a{\alpha}               \def\g{\gamma}
           \def\e{\varepsilon}
           
\def\s{\sigma}       \def\t{\theta}       
         \def\r{\rho}         \def\z{\zeta}

\begin{document}

\title{Remarks on one-component inner functions}


\thanks{This research was supported by Finnish Cultural Foundation.}



\author{Atte Reijonen}
\address{Graduate School of Information Sciences, Tohoku University, Aoba-ku, Sendai 980-8579, Japan}
\email{atte.reijonen@uef.fi}

\maketitle



\begin{abstract}
A one-component inner function $\Theta$ is an inner function whose level set
    $$\Omega_{\Theta}(\varepsilon)=\{z\in \mathbb{D}:|\Theta(z)|<\varepsilon\}$$
is connected for some $\varepsilon\in (0,1)$.
We give a sufficient condition for a Blaschke product with zeros in a Stolz domain to be
a one-component inner function.
Moreover, a sufficient condition is obtained in the case of atomic singular inner functions.

We study also derivatives of one-component inner functions in the Hardy and Bergman spaces.
For instance, it is shown that, for $0<p<\infty$, the derivative of a one-component inner function $\Theta$ is a member of the Hardy space $H^p$
if and only if $\Theta''$ belongs to the Bergman space $A_{p-1}^p$, or equivalently $\Theta'\in A_{p-1}^{2p}$.
\end{abstract}

\section{Examples of one-component inner functions}\label{Sec1}

Let $\D$ be the open unit disc of the complex plane $\C$.
A bounded and analytic function in $\D$ is an inner function if it has unimodular radial limits almost everywhere on the boundary $\T$ of $\D$. In this note, we study so-called one-component inner functions \cite{Cohn1982}, which are inner functions $\Th$ whose level set
    $$\Omega_\Th(\e)=\{z\in \D:|\Th(z)|<\e\}$$
is connected for some $\e\in (0,1)$. In particular, Blaschke products in this class are of interest. For a given sequence $\{z_n\}\subset \D \setminus \{0\}$ satisfying
$\sum_{n}(1-|z_n|)<\infty$, the Blaschke product with zeros $\{z_n\}$ is defined by
    \begin{equation*}\label{Eq:Blaschke}
    B(z)=\prod_{n}\frac{|z_n|}{z_n}\frac{z_n-z}{1-\overline{z}_nz}, \quad z\in \D.
    \end{equation*}
Here each zero $z_n$ is repeated according to its multiplicity. In addition, we assume that $\{z_n\}$ is ordered by non-decreasing moduli.

Recently several authors have studied one-component inner functions in the context of model spaces and operator theory;
see for instance \cite{ALMP2016, BBK2011, Besonov2016, Besonov2015}. In addition, A.~B.~Aleksandrov's paper \cite{Aleksandrov2002}, which contains several characterizations for one-component inner functions, is worth mentioning.
These references do not offer any concrete examples of infinite one-component Blaschke products.
In recent paper \cite{CM2017} by J.~Cima and R.~Mortini, one can find some examples. However, all one-component Blaschke products
constructed in \cite{CM2017} have some heavy restrictions. Roughly speaking, zeros of all of them are at least uniformly separated.
Recall that $\{z_n\}\subset \D$ is called uniformly separated if
    $$
    \inf_{n\in\N}\prod_{k\ne n}\left|\frac{z_k-z_n}{1-\overline{z}_kz_n}\right|>0.
    $$
As a concrete example, we mention that the Blaschke product with zeros $z_n=1-2^{-n}$ for $n\in \N$ is a one-component inner function \cite{CM2017}.
In addition, it is a well-known fact that every finite Blaschke product is a one-component inner function.

For $\g\ge 1$, $\xi\in\T$ and $C>0$, we define
    $$
    R(\g,\xi,C)=\{z\in\D:|1-\overline{\xi} z|^\g\le C(1-|z|)\}.
    $$
The region $R(1,\xi,C)$ is a Stolz domain with vertex at $\xi$.
Note that in the case $\g=1$ we have to assume $C>1$.
For $\g>1$, $R(\g,\xi,C)$ is a tangential approaching region in $\D$,
which touches $\T$ at $\xi$.
Denote by $\mathcal{R}_\g$ the family of all Blaschke products whose zeros lie in some $R(\g,C,\xi)$ with a fixed $\g$.
References related to $\mathcal{R}_\g$ are for instance \cite{AhernClark1974,Cargo1962,GPV2008}.
With these preparations we are ready to state our first main result.

\begin{theorem}\label{Thm1}
Let $B$ be a member of $\mathcal{R}_1$ with zeros $\{z_n\}_{n=1}^\infty$.
If
    \begin{equation}\label{Eq:thm1-cond}
    \begin{split}
    \liminf_{n \rightarrow \infty}\frac{\sum_{|z_j|>|z_n|} (1-|z_j|)}{1-|z_n|}>0,
    \end{split}
    \end{equation}
then $B$ is a one-component inner function.
\end{theorem}

As a consequence of Theorem~\ref{Thm1}, we obtain the affirmative
answer to the following question posed in \cite{CM2017}:
Is the Blaschke product $B$ with zeros $z_n=1-n^{-2}$ for $n\in \N$ a one-component inner function?
Some other examples of one-component inner functions are listed below. All of these examples can be verified
by using the fact that condition \eqref{Eq:thm1-cond} is valid if $\{z_n\}$ is ordered by strictly increasing moduli and
    $$ \liminf_{n \rightarrow \infty} \frac{1-|z_{n+1}|}{1-|z_n|}>0.$$

\begin{example}\label{Coro1}
Let $1<\a<\infty$ and $B$ be a Blaschke product with zeros
    \begin{itemize}
    \item[\rm(a)] $z_n=1-n^{-\a}$ for $n\in \N$, or
    \item[\rm(b)] $z_n=1-\frac{1}{n(\log n)^\a}$ for $n\in \N\setminus \{1\}$, or
    \item[\rm(c)] $z_n=1-\a^{-n}$ for $n\in \N$.
    \end{itemize}
Then $B$ is a one-component inner function.
\end{example}

A Blaschke product $B$ is said to be thin if its zeros $\{z_n\}_{n=1}^\infty$ satisfy
    $$\lim_{n \rightarrow \infty} (1-|z_n|^2)|B'(z_n)|=1.$$
We interpret that finite Blaschke products are not thin.
By \cite[Corollary~21]{CM2017}, any thin Blaschke product is not a one-component inner function.
Using this fact and \cite[Proposition~4.3(i)]{CFT2003}, we can give an example which shows that
condition \eqref{Eq:thm1-cond} in Theorem~\ref{Thm1} is essential.

\begin{example}
Let $B$ be the Blaschke product with zeros $\{w_n\}_{n=1}^\infty$ ordered by strictly increasing moduli and satisfying
    $$\frac{1-|w_{n+1}|}{1-|w_n|}\longrightarrow 0, \quad n \rightarrow \infty.$$
Then, by \cite[Proposition~4.3(i)]{CFT2003}, $B$ is a thin Blaschke product (with uniformly separated zeros).
Consequently, for instance, the Blaschke product with zeros $z_n=1-2^{-2^n}$ for $n\in \N$ is not a one-component inner function.
Note that zeros $\{z_n\}$ lie in $R(1,1,C)$ for every $C>1$ but they do not satisfy \eqref{Eq:thm1-cond}.
\end{example}

Let us recall a classical result of O.~Frostman \cite{Frostman1942}:
The Blaschke product $B$ with zeros $\{z_n\}$ has a unimodular radial limit at $\xi \in \T$ if and only if
    \begin{equation} \label{Eq:Frostman-cond}
    \sum_n \frac{1-|z_n|}{|\xi-z_n|}<\infty.
    \end{equation}
A Blaschke product is called a Frostman Blaschke product if it has a unimodular radial limit at every point on $\T$.
It is a well-known fact that an infinite Frostman Blaschke product cannot be a one-component inner function; see for instance
\cite[Theorem~1.11]{Aleksandrov2002} or Theorem~\ref{ThmB} in Section~\ref{Sec3}.
Using this fact, we show that any $\mathcal{R}_\g$ with $\g>1$ contains
a member which is not a one-component inner function but its zeros $\{z_n\}$ satisfy \eqref{Eq:thm1-cond}.
This means that the hypothesis $B\in \mathcal{R}_1$ in Theorem~\ref{Thm1} is essential.

\begin{example}
Fix $\g>1$ and choose $\a=\a(\g)>1$ such that $\a>\frac{\g}{\g-1}$. Let $\{z_n\}$ be such that
    $$|z_n|=1-n^{-\a} \quad \text{and} \quad |1-z_n|=n^{-\a/\g}, \quad n\in \N.$$
Since the sequence $\{z_n\}$ is a subset of $R(\g,1,1)$, all points of $\{z_n\}$ lie in $\D$.
Moreover, it is clear that $\{z_n\}$ satisfies the Blaschke condition $\sum_n (1-|z_n|)<\infty$
and \eqref{Eq:thm1-cond} in Theorem~\ref{Thm1}.
Hence the Blaschke product with zeros $\{z_n\}$ is well-defined. Furthermore,
    \begin{equation*}
    \sum_{n=1}^\infty \frac{1-|z_n|}{|1-z_n|}
    =\sum_{n=1}^\infty n^{\a/\g-\a}<\infty;
    \end{equation*}
and thus, the Blaschke product $B$ has a unimodular radial limit at 1 by Frostman's result.
Since condition \eqref{Eq:Frostman-cond} is trivially valid for every $\xi \in \T\setminus \{1\}$,
$B$ is an infinite Frostman Blaschke product. Consequently, it is not a one-component inner function.
\end{example}

Recall that a singular inner function takes the form
    $$
    S_\s(z)=\exp\left(\int_{\T} \frac{z+\xi}{z-\xi}\, d\s(\xi)\right),\quad z\in\D,
    $$
where $\s$ is a positive measure on $\T$, singular with respect to the Lebesgue measure.
If the measure $\s$ is atomic, then this definition reduces to the form
    \begin{equation*}
    S(z)=\exp\left(\sum_{n}\gamma_n \frac{z+e^{i\t_n}}{z-e^{i\t_n}} \right),\quad z\in\D,
    \end{equation*}
where $\t_n\in [0,2\pi)$ are distinct points and $\gamma_n>0$ satisfy $\sum_{n}\g_n<\infty$. These functions are known as atomic singular inner functions associated with $\{e^{i\t_n}\}$ and $\{\gamma_n\}$.

An atomic singular inner function associated with a measure having only finitely many mass points is a one-component inner function; see \cite[Corollary~17]{CM2017}. In the literature, one cannot find any example of a one-component singular inner function associated with a measure having infinitely many mass points.
However, the following result gives a way to construct such functions.

\begin{theorem}\label{Thm1b}
Let $S$ be the atomic singular inner function associated with $\{e^{i\t_n}\}_{n=0}^\infty$ and $\{\gamma_n\}_{n=0}^\infty$.
Moreover, assume that the following conditions are valid:
    \begin{itemize}
    \item[\rm(i)] $\t_0=0$, $\{\t_n\}_{n=1}^\infty \subset (0,1)$ is strictly decreasing and $\lim_{n \rightarrow \infty} \t_n=0$.
    \item[\rm(ii)] There exists a constant $C=C(S)>0$ such that $|\t_{n-1}-\t_{n+1}|\le C\g_n^2$
    for all sufficiently large $n\in \N$.
    \end{itemize}
Then $S$ is a one-component inner function.
\end{theorem}

Next we give a concrete example of a one-component singular inner function.
This example is a direct consequence of Theorem~\ref{Thm1b}.

\begin{example}\label{Ex-sing}
Let $\t_0=0$, $\t_n=2^{-n}$, $\g_0=1$ and $\g_n=n^{-2}$ for $n\in \N$. Then the atomic singular inner function $S$
associated with $\{e^{i\t_n}\}_{n=0}^\infty$ and $\{\g_n\}_{n=0}^\infty$ is a one-component inner function.
\end{example}

The remainder of this note is organized as follows.
Sections~\ref{Sec3}~and~\ref{Sec4} consist of the proofs of Theorems~\ref{Thm1}~and~\ref{Thm1b}, respectively.
In Section~\ref{Sec2}, we study one-component inner functions
whose derivatives belong to the Hardy or Bergman spaces. In particular, we give partial improvements
for \cite[Theorem~6.2]{Ahern1979} and \cite[Theorem~3.10]{GGJ2011}.

\section{Proof of Theorem~\ref{Thm1}}\label{Sec3}

We begin by stating a modification of \cite[Theorem~1.11]{Aleksandrov2002}, which is due to \cite[p.~2915,~Remark~2]{Aleksandrov2002}.
This result offers two practical characterizations for one-component inner functions
and plays an important role in the proofs of Theorems~\ref{Thm1}~and~\ref{Thm1b}.
Before it we recall that the spectrum $\r(\Th)$ of an inner function $\Th$ is the set of all point on $\T$ in which $\Th$ does not have an
analytic continuation. It is a well-known fact that the spectrum of a Blaschke product consists of the accumulation points of zeros.
By \cite[Chapter~2,~Theorem~6.2]{Garnett1981}, the spectrum of a singular inner function $S_\s$ is the closed support of the associated measure $\s$.

\begin{lettertheorem}\label{ThmB}
Let $\Th$ be an inner function. Then the following statements are equivalent:
    \begin{itemize}
    \item[\rm(a)] $\Th$ is a one component inner function.
    \item[\rm(b)] There exists a constant $C=C(\Th)>0$ such that
    \begin{equation}\label{Eq:lA-2}
    |\Th''(\z)|\le C|\Th'(\z)|^2, \quad \z\in \T\setminus \r(\Th),
    \end{equation}
and
    \begin{equation}\label{Eq:lA-1}
    \liminf_{r \rightarrow 1^{-}} |\Th(r\xi)|<1, \quad \xi\in \r(\Th).
    \end{equation}
\item[\rm(c)] There exists a constant $C=C(\Th)>0$ such that \eqref{Eq:lA-2} holds,
the Lebesgue measure of $\r(\Th)$ is zero and $\Th'$ is not bounded on any arc $\Gamma \subset \T\setminus \r(\Th)$
with $\overline{\Gamma} \cap \r(\Th)\neq \emptyset$.
    \end{itemize}
\end{lettertheorem}

Write $f\lesssim g$ if there exists a constant $C>0$ such that $f\le Cg$, while $f\gtrsim g$ is understood in an analogous manner. If $f\lesssim g$ and $f\gtrsim g$, then the notation $f\asymp g$ is used.
With these preparations we are ready to prove Theorem~\ref{Thm1}.

\medskip

\noindent
\emph{Proof of Theorem~\ref{Thm1}.}
If $B$ is an arbitrary Blaschke product with zeros $\{z_n\}$, then
    \begin{equation*}
    \frac{B'(z)}{B(z)}=\sum_{n=1}^\infty \frac{|z_n|^2-1}{(1-\overline{z}_nz)(z_n-z)}\quad \text{and} \quad |B(z)|\le \frac{|z_n-z|}{|1-\overline{z}_nz|}.
    \end{equation*}
Using these estimates, one can easily verify
    \begin{equation*}
    |B''(z)|\le \frac{|B'(z)|^2}{|B(z)|}+2|B(z)|^{-1}\sum_{n=1}^\infty \frac{1-|z_n|^2}{|1-\overline{z}_nz|^3}, \quad z\in \D.
    \end{equation*}
In particular,
    \begin{equation*}
    \begin{split}
    |B''(\z)|&\le |B'(\z)|^2+2\sum_{n=1}^\infty \frac{1-|z_n|^2}{|1-\overline{z}_n\z|^3}
    \asymp |B'(\z)|^2+\sum_{n=1}^\infty \frac{1-|z_n|^2}{|z_n-\z|^3}
    \end{split}
    \end{equation*}
for every $\z \in \T\setminus \r(B)$. Using \cite[Theorem~2]{AhernClark1974}, we deduce that \eqref{Eq:lA-2} (with $\Th=B$)
is valid if
     \begin{equation}\label{Eq:sum-cond}
     \begin{split}
     \sum_{n=1}^\infty \frac{1-|z_n|^2}{|z_n-\z|^3}\lesssim \left(\sum_{n=1}^\infty \frac{1-|z_n|^2}{|z_n-\z|^2}\right)^2
     \end{split}
     \end{equation}
holds for every $\z \in \T\setminus \r(B)$.

Assume without loss of generality that zeros $\{z_n\}$ of $B$ lie in a Stolz domain $R(1,1,C)$, and
remind that $\{z_n\}$ is ordered by non-decreasing moduli.
Then $\r(B)=\{1\}$, and the
functions $f$ and $g$, defined by
    \begin{equation*}
    f(x)=
    \left\{ \begin{array}{rl}
    1, & \quad 0\le x<1, \\
    \min_{n\le x}(1-|z_n|), & \quad 1\le x<\infty,
    \end{array} \right.
    \end{equation*}
and
    $$g(\t)=\inf\{x:f(x)\le \t\}, \quad 0<\t\le 1,$$
are non-increasing. Since $f(w)=1-|z_n|$ for $n\in\N$ and $n\le w<n+1$,
it is clear that $g:(0,1] \rightarrow \N\cup \{0\}$, $f(x)>\t$ for $x<g(\t)$, and $f(x)\le \t$ for $x\ge g(\t)$.
Write $\z=e^{i\t}$, and assume without loss of generality that $\t>0$ is close enough to zero.
Using \cite[Lemma~3]{AhernClark1974} together with some standard estimates, we obtain
    \begin{equation}\label{Eq:proof-thm1-1}
    \begin{split}
    \left(\sum_{n=1}^\infty \frac{1-|z_n|^2}{|z_n-\z|^3}\right)^{1/2}
    &\asymp \left(\sum_{n=1}^\infty \frac{1-|z_n|^2}{||z_n|-\z|^3}\right)^{1/2}
    \asymp \left(\sum_{n=1}^\infty \frac{1-|z_n|}{[(1-|z_n|)^2+\t^2]^{3/2}}\right)^{1/2} \\
    &\le \left(\sum_{n<g(\t)} f(n)^{-2} + \t^{-3}\sum_{n\ge g(\t)} f(n)\right)^{1/2} \\
    &\le \left(\sum_{n<g(\t)} f(n)^{-2}\right)^{1/2} + \t^{-3/2}\left(\sum_{n\ge g(\t)} f(n)\right)^{1/2} \\
    &\le \sum_{n<g(\t)} f(n)^{-1} + \t^{-3/2}\left(\sum_{n\ge g(\t)} f(n)\right)^{1/2}.
    \end{split}
    \end{equation}
Applying hypothesis \eqref{Eq:thm1-cond} and the above-mentioned properties of $f$ and $g$, we find $C=C(B)>0$ such that
    $$\sum_{n\ge g(\t)} f(n)\ge Cf(g(\t)-1)\ge C\t.$$
It follows that
    \begin{equation}\label{Eq:proof-thm1-2}
    \begin{split}
    \sum_{n=1}^\infty \frac{1-|z_n|^2}{|z_n-\z|^2}&\asymp \sum_{n=1}^\infty \frac{1-|z_n|}{(1-|z_n|)^2+\t^2} \\
    &\ge \frac{1}{2}\sum_{n<g(\t)} f(n)^{-1} + \frac{\t^{-2}}{2}\sum_{n\ge g(\t)} f(n) \\
    &\ge  \frac{1}{2}\sum_{n<g(\t)} f(n)^{-1} + \frac{\sqrt{C}\t^{-3/2}}{2}\left(\sum_{n\ge g(\t)} f(n)\right)^{1/2}.
    \end{split}
    \end{equation}
Using estimates \eqref{Eq:proof-thm1-1} and \eqref{Eq:proof-thm1-2}, it is easy to see that condition \eqref{Eq:sum-cond} is
valid for $\z \in \T\setminus \{1\}$. Consequently, $B$ satisfies \eqref{Eq:lA-2}.

Let $B_0$ be the Blaschke product with zeros $\{|z_n|\}$.
It is obvious that $\liminf_{r \rightarrow 1^{-}} |B_0(r)|=0$.
Hence, by the deduction above, it is clear that $B_0$ satisfies condition (b) in Theorem~\ref{ThmB}, and thus also the other conditions are valid.
Since $B_0$ satisfies (c) in Theorem~\ref{ThmB}, also $B$ satisfies it. This is due to \cite[Lemma~3]{AhernClark1974},
which asserts that $|B'(\xi)|\asymp |B_0'(\xi)|$ for $\xi \in \T\setminus \{1\}$.
Hence $B$ is a one-component inner function by Theorem~\ref{ThmB}. This completes the proof. \hfill$\Box$

\section{Proof of Theorem~\ref{Thm1b}}\label{Sec4}

Let us prove Theorem~\ref{Thm1b}.

\medskip

\noindent
\emph{Proof of Theorem~\ref{Thm1b}.} Due to hypothesis (i), the set of mass points $\{e^{i\t_n}\}_{n=0}^\infty$ is closed. Consequently,
the spectrum $\r(S)$ consists of points $\{e^{i\t_n}\}_{n=0}^\infty$. Hence, by \cite[Chapter~2, Theorem~6.2]{Garnett1981}, we have
    $$\lim_{r \rightarrow 1^{-}} |S(r\xi)|=0, \quad \xi\in \r(S).$$
This means that $S$ satisfies condition \eqref{Eq:lA-1} (with $\Th=S$) in Theorem~\ref{ThmB}. Consequently, it suffices to show that $S$ fulfills also \eqref{Eq:lA-2}.

By a straightforward calculation, one can check that
    \begin{equation*}
    \begin{split}
    S''(z)=4\left(\sum_{n=0}^\infty \frac{\g_n e^{i\t_n}}{(z-e^{i\t_n})^3}+\left(\sum_{m=0}^\infty \frac{\g_m e^{i\t_m}}{(z-e^{i\t_m})^2}\right)^2\right)
    \exp\left(\sum_{k=0}^\infty \g_k\frac{z+e^{i\t_k}}{z-e^{i\t_k}}\right), \quad z\in \D.
    \end{split}
    \end{equation*}
Since
    $$|S'(\z)|=2\sum_{n=0}^\infty \frac{\g_n}{|\z-e^{i\t_n}|^2}, \quad \z\in \T\setminus \r(S),$$
by \cite[Theorem~2]{AhernClark1974}, we obtain
    \begin{equation*}
    \begin{split}
    |S''(\z)|\le 4\sum_{n=0}^\infty \frac{\g_n}{|\z-e^{i\t_n}|^3}+|S'(\z)|^2, \quad \z\in \T\setminus \r(S).
    \end{split}
    \end{equation*}
Consequently, it suffices to show
    \begin{equation}\label{Eq:thm5-proof-0}
    \begin{split}
    \sum_{n=0}^\infty \frac{\g_n}{|\z-e^{i\t_n}|^3}
    \lesssim \left(\sum_{n=0}^\infty \frac{\g_n}{|\z-e^{i\t_n}|^2}\right)^2, \quad \z\in \T\setminus \r(S).
    \end{split}
    \end{equation}

Assume without loss of generality that $\z \in \T\setminus \r(S)$ is close enough to one, and write $\z=e^{i\t}$.
Choose $j=j(\t,S)\in \N\cup \{0\}$ such that $|\t-\t_j|$ is as small as possible.
Then standard estimates yield
    \begin{equation}\label{Eq:thm5-proof-1}
    \begin{split}
    \left(\sum_{n=0}^\infty \frac{\g_n}{|\z-e^{i\t_n}|^3}\right)^{1/2}
    &\asymp \left(\sum_{n=0}^\infty \frac{\g_n}{|\t-\t_n|^3}\right)^{1/2} \\
    &\le |\t-\t_j|^{-3/2}  \left(\sum_{n=0}^\infty \g_n\right)^{1/2} \asymp |\t-\t_j|^{-3/2}
    \end{split}
    \end{equation}
and
    \begin{equation}\label{Eq:thm5-proof-2}
    \begin{split}
    \sum_{n=0}^\infty \frac{\g_n}{|\z-e^{i\t_n}|^2}\asymp \sum_{n=0}^\infty \frac{\g_n}{|\t-\t_n|^2}
    \ge \frac{\g_j}{|\t-\t_j|^2}.
    \end{split}
    \end{equation}
If $\t<0$, then $j=0$; and hence, \eqref{Eq:thm5-proof-0} is a direct consequence of \eqref{Eq:thm5-proof-1} and \eqref{Eq:thm5-proof-2}.
Let $\t>0$. By hypothesis (i), we have $\t_{j+1}<\t<\t_{j-1}$, where $j\in \N$ is large enough. Consequently, hypothesis (ii) gives
    $$\frac{\g_j}{|\t-\t_j|^2}\ge \frac{\g_j}{|\t-\t_j|^{3/2}|\t_{j-1}-\t_{j+1}|^{1/2}}\gtrsim |\t-\t_j|^{-3/2}.$$
According to this estimate, \eqref{Eq:thm5-proof-0} is a consequence of \eqref{Eq:thm5-proof-1} and \eqref{Eq:thm5-proof-2}. Finally the assertion follows from Theorem~\ref{ThmB}. \hfill$\Box$

\section{Derivatives of one-component inner functions in functions spaces}\label{Sec2}

We begin by fixing the notation.
Let $\HO(\D)$ be the space of all analytic functions in $\D$.
For $0<p<\infty$, the Hardy space $H^p$ consists of those $f\in\HO(\D)$ such that
    \begin{equation*}
    \begin{split}
    \|f\|_{H^p}= \sup_{0\le r<1}M_p(r,f)<\infty, \quad \text{where} \quad M_p^p(r,f)=\frac{1}{2\pi}\int_{0}^{2\pi} |f(re^{i\t})|^p\,d\t.
    \end{split}
    \end{equation*}
For $0<p<\infty$ and $-1<\a<\infty$, the Bergman space $A_\a^p$ consists of those $f\in\HO(\D)$ such that
    $$
    \|f\|_{A^p_\a}^p=\int_\D|f(z)|^p(1-|z|)^\a\,dA(z)<\infty,
    $$
where $dA(z)=dx\,dy$ is the Lebesgue area measure on $\D$.

By \cite[Theorem~5]{LP1936} and \cite[Lemma~1.4]{Vinogradov}, we have
    \begin{equation}\label{Eq:inc1}
    \{f:f'\in A_{p-1}^p\}\subset H^p, \quad 0<p\le 2.
    \end{equation}
and
    \begin{equation}\label{Eq:inc2}
    H^p \subset \{f:f'\in A_{p-1}^p\}, \quad 2\le p<\infty.
    \end{equation}
It is clear that $\{f:f'\in A_1^2\}=H^2$, while otherwise the inclusions are strict. For instance, an example showing the strictness of inclusions \eqref{Eq:inc1} and \eqref{Eq:inc2} can be given by using gap series; see details in \cite{BGP2004}. Nevertheless, we have the
following result, which is essentially a consequence of \cite[Theorem~6.2]{Ahern1979} and \cite[Theorem~3.10]{GGJ2011}.

\begin{theorem}\label{ThmA}
Let $\frac12<p<\infty$ and $\Th$ be an inner function. Then the following statements are equivalent:
    \begin{itemize}
    \item[\rm(a)] $\Th'\in H^p$,
    \item[\rm(b)] $\Th'\in A_{p-1}^{2p}$,
    \item[\rm(c)] $\Th''\in A_{p-1}^{p}$.
    \end{itemize}
\end{theorem}

Before the proof of Theorem~\ref{ThmA}, we note that, for  $f\in \HO(\D)$, $n\in \N$ and $0<p<\infty$,
we have $M_p(r,f^{(n)}) \asymp M_p(r,D^n f)$
with comparison constants independent of $r$ \cite{Flett1972}.
Here $D^n$ is the fractional derivative of order $n$.
This fact is exploited when we apply some results in the literature.

\begin{proof}
The equivalence $\rm (a)\Leftrightarrow (c)$ is a consequence of \cite[Theorem~3.10]{GGJ2011}.
For $\frac12<p<1$, the equivalence $\rm (a)\Leftrightarrow (b)$ can be verified, for instance,
using \cite[Theorem~6.2]{Ahern1979} together with \cite[Corollary~7]{PGRR2016}.
It is a well-known fact the only inner functions whose derivative belongs to $H^p$ for some $p\ge 1$ are finite Blaschke products.
Using this fact together with \cite[Theorem~7(c)]{Gluchoff1987} and the equivalence $\rm (a)\Leftrightarrow (c)$, it is easy to deduce that
an inner function $\Th$ is a finite Blaschke product if it satisfies any of conditions (a)--(c) for some $p\ge 1$.
In addition, it is clear that every finite Blaschke product $\Th$ satisfies conditions (a)--(c) for all $p>0$.
Finally the assertion follows by combining the above-mentioned facts.
\end{proof}

By \cite[Lemma~2]{AhernClark1976}, there exists a Blaschke product $B$ such that $B'\in A_{-1/2}^1 \setminus H^{1/2}$.
This means that, for $p=\frac12$, condition (b) in Theorem~\ref{ThmA} does not always imply (a).
Nevertheless, it is an open question whether the equivalence $\rm (a)\Leftrightarrow (c)$ is valid also for $0<p\le \frac12$.
This question was earlier posed in \cite{RS2018}.
The next result shows that the statement of Theorem~\ref{ThmA} is valid for all $p>0$ if $\Th$ is a one-component inner function.
Consequently, we obtain a partial answer to the question.

\begin{theorem}\label{Thm2}
Let $0<p<\infty$ and $\Th$ be a one-component inner function.  Then conditions \emph{(a)--(c)} in Theorem~\ref{ThmA} are equivalent.
\end{theorem}

Next we recall \cite[Theorem~1.9]{Aleksandrov2002}, which consists of a strengthened Schwarz-Pick lemma for
one-component inner functions. This result plays a key role in the proof of Theorem~\ref{Thm2}.

\begin{lettertheorem}\label{ThmC}
Let $n\in \N$ and $\Th$ be a one-component inner function. Then there exists $C=C(n,\Th)>0$ such that
    \begin{equation}\label{Eq:high-SP}
    |\Th^{(n)}(z)|\le C\left(\frac{1-|\Th(z)|}{1-|z|}\right)^n
    \end{equation}
for all $z\in \D$.
\end{lettertheorem}

For the proof of Theorem~\ref{Thm2}, we need also a generalization of \cite[Theorem~6.1]{Ahern1979}.

\begin{lemma}\label{Lemma1}
Let $0<p<1$, $-1<\a<\infty$ and $\Th$ be an inner function. Then there exists $C=C(p,\a)>0$ such that
    \begin{equation*}
    \int_0^1 |\Th'(re^{i\t})|^{p+\a+1}(1-r)^\a dr \le C|\Th'(e^{i\t})|^{p}, \quad e^{i\t}\in \T\setminus \r(\Th).
    \end{equation*}
In particular, $\|\Th'\|_{A_\a^{p+\a+1}}^{p+\a+1} \le 2\pi C\|\Th'\|_{H^p}^p$.
\end{lemma}

\begin{proof}
Let $e^{i\t}\in \T\setminus \r(\Th)$.
By \cite[Lemma~6.1]{Ahern1979}, we know that $|\Th'(re^{i\t})|\le 4|\Th'(e^{i\t})|$ for all $r\in [0,1)$.
Using this fact together with the Schwarz-Pick lemma, we obtain
    \begin{equation*}
    \begin{split}
    \int_0^1 |\Th'(re^{i\t})|^{p+\a+1}(1-r)^\a dr &\le\int_0^{x}(1-r)^{-p-1}dr + (4|\Th'(e^{i\t})|)^{p+\a+1}\int_x^1 (1-r)^\a dr \\
    &\lesssim (1-x)^{-p}-1+(1-x)^{\a+1}|\Th'(e^{i\t})|^{p+\a+1}
    \end{split}
    \end{equation*}
for every $x\in [0,1]$. Now it suffices to show that
    \begin{equation}\label{Eq:lemma1-1}
    \begin{split}
    (1-x)^{-p}-1+(1-x)^{\a+1}|\Th'(e^{i\t})|^{p+\a+1}\lesssim |\Th'(e^{i\t})|^p
    \end{split}
    \end{equation}
for some $x$.
If $|\Th'(e^{i\t})|\le 1$, then this true for $x=0$. In the case where $|\Th'(e^{i\t})|>1$, the
choice $x=1-1/|\Th'(e^{i\t})|$ implies \eqref{Eq:lemma1-1}. Since the last assertion is a direct
consequence of the first assertion, Hardy's convexity and the mean convergence theorem \cite{Duren1970}, the proof is complete.
\end{proof}

Now we are ready to prove Theorem~\ref{Thm2}.

\medskip

\noindent
\emph{Proof of Theorem~\ref{Thm2}.}
By Theorem~\ref{ThmA}, we may assume $0<p<1$ (or even $p\le \frac12$).
Using Theorem~\ref{ThmC} with $n=2$, \cite[Theorem~6]{Ahern1983} and Lemma~\ref{Lemma1} with $\a=p-1$, we obtain
    \begin{equation}\label{Eq:Thm2_1}
    \begin{split}
    \|\Th''\|_{A_{p-1}^p}^p &\lesssim \int_{\D}  \left(\frac{1-|\Th(z)|}{1-|z|}\right)^{2p}(1-|z|)^{p-1}dA(z)
    \asymp \|\Th'\|_{A_{p-1}^{2p}}^{2p}\lesssim  \|\Th'\|_{H^p}^{p}.
    \end{split}
    \end{equation}
The assertion follows from \eqref{Eq:inc1} and \eqref{Eq:Thm2_1}. \hfill$\Box$

\medskip

By the proof of Theorem~\ref{Thm2}, it is obvious that conditions (a)--(c) in Theorem~\ref{ThmA} are equivalent
also if $p>0$ and $\Th$ is an inner function satisfying \eqref{Eq:high-SP} with $n=2$.
Note that, using this observation and keeping mind the discussion after the proof of Theorem~\ref{ThmA},
we can find an inner function such that \eqref{Eq:high-SP} with $n=2$ is not valid for any $C>0$.

It is a well-known fact that, for $0<p<\infty$ and $-1<\a<\infty$,
the Bergman space $A_\a^p$ coincides with $\{f:f'\in A_{\a+p}^p\}$ \cite{Flett1972}.
Using this result, it is easy to generalize condition (c) in Theorem~\ref{Thm2} to the form
$\Th^{(n)}\in A_{p(n-1)-1}^p$ for any/every $n\in \N\setminus \{1\}$. For $0<p<1$, we
can show this also by modifying the proof of Theorem~~\ref{Thm2}; and as a substitute of
this process we obtain an alternative version of Theorem~\ref{Thm2}.

\begin{theorem}\label{Coro-move}
Let $0<p<\infty$ and $\Th$ be a one-component inner function. Then the following statements are equivalent:
    \begin{itemize}
    \item[\rm(a)] $\Th'\in H^p$,
    \item[\rm(b)] $\Th'\in A_{\a}^{p+\a+1}$ for some $\a\in (-1,\infty)$,
    \item[\rm(c)] $\Th'\in A_{\a}^{p+\a+1}$ for every $\a\in (-1,\infty)$.
    \end{itemize}
\end{theorem}

\begin{proof}
By the proof of Theorem~\ref{ThmA}, we know that, for $1\le p<\infty$, $\Th$ satisfies any/all of conditions (a)--(c) if and only if
it is a finite Blaschke product. Hence we may assume $0<p<1$. Moreover, let $-1<\a<\infty$ and $n\in \N\setminus \{1\}$.
Then \cite[Theorem~3]{Flett1972}, Theorem~\ref{ThmC}, \cite[Theorem~6]{Ahern1983}, the Schwarz-Pick lemma and Lemma~\ref{Lemma1} yield
   \begin{equation}\label{Eq:Coro-move1}
    \begin{split}
    \|\Th'\|_{H^p}^p &\lesssim \|\Th^{(n)}\|_{A_{p(n-1)-1}^p}^p \lesssim \int_{\D}  \left(\frac{1-|\Th(z)|}{1-|z|}\right)^{np}(1-|z|)^{p(n-1)-1}dA(z) \\
    &\asymp \|\Th'\|_{A_{p(n-1)-1}^{np}}^{np}\le \|\Th'\|_{A_{\a}^{p+\a+1}}^{p+\a+1} \lesssim  \|\Th'\|_{H^p}^{p}
    \end{split}
    \end{equation}
when $p+\a+1 \le np$. Since we may choose $n=n(p,\a)$ such that $n\ge (p+\a+1)/p$, the assertion follows from \eqref{Eq:Coro-move1}.
\end{proof}

Note that, applying Theorem~\ref{Coro-move} and \cite[Theorem~3]{RS2018},
we can give several characterizations for one-component inner functions $\Th$ whose derivative belongs to $H^p$
for some $p\in (0,1)$. By \cite[Corollary~4]{RS2018}, these characterizations for $p\in (\frac12,1)$
are valid even if $\Th$ would be an arbitrary inner function.
Next we show a counterpart of Theorem~\ref{Coro-move} for all members of $\mathcal{R}_1$.

\begin{corollary}\label{Coro:stolz-move}
Let $0<p<\infty$ and $B\in \mathcal{R}_1$. Then the following statements are equivalent:
    \begin{itemize}
    \item[\rm(a)] $B'\in H^p$,
    \item[\rm(b)] $B'\in A_{\a}^{p+\a+1}$ for some $\a\in (-1,\infty)$,
    \item[\rm(c)] $B'\in A_{\a}^{p+\a+1}$ for every $\a\in (-1,\infty)$.
    \end{itemize}
\end{corollary}

\begin{proof}
Assume without loss of generality that $0<p<1$ and zeros $\{w_n\}$ of $B$ lie in a Stolz domain $R(1,1,C)$.
Let $B_0$ be the Blaschke product with zeros $x_n=1-2^{-n}$ for $n\in \N$, write $\Th=BB_0$ and $\{z_n\}=\{w_n\} \cup \{x_n\}$,
where $\{z_n\}$ is ordered by non-decreasing moduli.
Then, for each $n\in \N$, there exists $k_n\in \N$ such that
    $$|x_{k_n}|\le |z_n|<|x_{k_n+1}|.$$
It follows that
    \begin{equation*}
    \begin{split}
    \frac{\sum_{|z_j|>|z_n|}(1-|z_j|)}{1-|z_n|}\ge \frac{1-|x_{k_n+1}|}{1-|x_{k_n}|}=\frac{1}{2}.
    \end{split}
    \end{equation*}
Consequently, $\Th$ is a one-component inner function by Theorem~\ref{Thm1}.

Let $-1<\a<\infty$.
By \cite[Theorem~5]{AhernClark1974} and a simple modification of \cite[Corollary~2.5]{PR2009} based on \cite[Theorem~6]{Ahern1983}, we know that
    $$\Th'\in H^p \quad \Leftrightarrow \quad B'\in H^p \quad \text{and} \quad B_0'\in H^p$$
and
    $$\Th'\in A_{\a}^{p+\a+1} \quad \Leftrightarrow \quad B'\in A_{\a}^{p+\a+1} \quad \text{and} \quad B_0'\in A_{\a}^{p+\a+1}.$$
In addition, \cite[Theorem~7]{AhernClark1974} and Lemma~\ref{Lemma1} imply $B_0'\in H^p \cap A_{\a}^{p+\a+1}$.
Finally, the assertion follows by using Theorem~\ref{Coro-move} together with the above-mentioned facts.
\end{proof}

Since the derivative of an arbitrary $B\in \mathcal{R}_1$ belongs to $H^p \cap A_{\a}^{p+\a+1}$
for every $p\in (0,\frac12)$ and $\a\in (-1,\infty)$ by \cite[Theorem~2.3]{GPV2007} and Lemma~\ref{Lemma1},
the statement of Corollary~\ref{Coro:stolz-move} for $p\neq \frac12$ does not come as a surprise. However, the case $p=\frac12$ is interesting
because it is not easy to find an alternative way to prove this result.

Recall that $f\in \HO(\D)$ belongs to the Nevalinna class $\mathcal{N}$ if
     $$\sup_{0\le r<1} \int_0^{2\pi} \log^+ |f(re^{i\t})|\,d\t<\infty,$$
where $\log^+ 0=0$ and $\log^+ x=\max\{0,\log x\}$ for $0<x<\infty$.
As a consequence of Theorem~\ref{ThmC}, we can also give sufficient conditions for higher order derivatives of one-component inner functions
to be in the Hardy space $H^p$ or Nevanlinna class $\mathcal{N}$.

\begin{corollary}\label{Coro-n-deri}
Let $0<p<\infty$, $n\in \N$ and $\Th$ be a one-component inner function. Then the following statements are valid:
    \begin{itemize}
    \item[\rm(a)] If $\Th'\in H^p$, then $\Th^{(n)}\in H^{p/n}$.
    \item[\rm(b)] If $\Th'\in \mathcal{N}$, then $\Th^{(n)}\in \mathcal{N}$.
    \end{itemize}
\end{corollary}

\begin{proof}
As a consequence of Theorem~\ref{ThmC} \cite{Aleksandrov2002}, we find $C=C(n,\Th)$ such that
    \begin{equation}\label{Eq:n-hardy-1}
    |\Th^{(n)}(\xi)|\le C|\Th'(\xi)|^n, \quad \xi\in \T\setminus \r(\Th).
    \end{equation}
Since the spectrum $\r(\Th)$ has a Lebesgue measure zero, inequality \eqref{Eq:n-hardy-1},
Hardy's convexity and the mean convergence theorem yield
    $$\|\Th^{(n)}\|_{H^{p/n}}^{p/n}=\frac{1}{2\pi}\int_0^{2\pi} |\Th^{(n)}(e^{i\t})|^{p/n} d\t
    \lesssim \frac{1}{2\pi}\int_0^{2\pi} |\Th'(e^{i\t})|^{p} d\t =\|\Th'\|_{H^p}^p.
    $$
Hence assertion (a) is proved. Since case (b) can be verified in a similar manner, the proof is complete.
\end{proof}

We close this note with two results regarding certain one-component singular inner functions.

\begin{corollary}\label{Coro-thm2}
Let $0<p<\infty$ and $S$ be the one-component atomic singular inner function associated with $\{e^{i\t_n}\}$ and $\{\gamma_n\}\in l^{1/2}$.
Then $S$ satisfies any/all of conditions \emph{(a)--(c)} in Theorem~\ref{ThmA} if and only if $p<\frac{1}{2}$.
\end{corollary}

\begin{proof}
By \cite[Theorem~3]{PGRR2016}, for $\frac14\le p<\infty$, the derivative of $S$ belongs to $A_{p-1}^{2p}$ if and only if $p<\frac12$.
Since $H^{p_1} \subset H^{p_2}$ for $0<p_2\le p_1<\infty$, the assertion follows from this result and Theorem~\ref{Thm2}.
\end{proof}

The following result shows that Corollary~\ref{Coro-n-deri}(a) is sharp.

\begin{corollary}\label{Coro-n-deri-2}
Let $0<p<\infty$, $m\in \N$ and $S$ be the one-component atomic singular inner function associated with $\{e^{i\t_n}\}$ and $\{\gamma_n\}\in l^{1/2}$.
Moreover, assume that there exist an index $j=j(S)$ and $\e=\e(j)>0$ such that $|\t_j-\t_n|>\e$ for all $n\neq j$.
Then $S^{(m)}\in H^p$ if and only if $p<\frac{1}{2m}$.
\end{corollary}

\begin{proof}
By Corollary~\ref{Coro-thm2}, $S'\in H^{mp}$ if and only if $p<\frac1{2m}$.
Consequently, Corollary~\ref{Coro-n-deri}(a) implies $S^{(m)}\in H^p$ for $p<\frac{1}{2m}$.
Hence it suffices to show that $S^{(m)}\in H^p$ only if $p<\frac{1}{2m}$.

Fix $j=j(S)$ to be the smallest index such that $|\t_j-\t_n|>\e$ for all $n\neq j$ and some $\e=\e(j)>0$.
Let us represent $S$ in the form $S=S_1S_2$, where
    $$S_1(z)=\exp\left(\g_j \frac{z+e^{i\t_j}}{z-e^{i\t_j}}\right), \quad z\in \D,$$
and $S_2=S/S_1$. Using this factorization, we obtain
    \begin{equation*}
    |S^{(m)}(e^{i\t})|=\left|\sum_{k=0}^m \binom{m}{k} S_1^{(m-k)}(e^{i\t})S_2^{(k)}(e^{i\t})\right|
    \asymp \left|S_1^{(m)}(e^{i\t})S_2(e^{i\t})\right|=|S_1^{(m)}(e^{i\t})|
    \end{equation*}
when $\t$ (which is not $\t_j$) is close enough to $\t_j$ depending on $S$ and $m$. Consequently, we find a sufficiently small $\a=\a(p,S,m)>0$ such that
    \begin{equation*}
    \int_0^{2\pi}|S_1^{(m)}(e^{i\t})|^p\,d\t \asymp \int_{\t_j-\a}^{\t_j+\a}|S_1^{(m)}(e^{i\t})|^p\,d\t
    \lesssim \int_0^{2\pi}|S^{(m)}(e^{i\t})|^p\,d\t,
    \end{equation*}
where the comparison constants depend only on $p$, $S$ and $m$. It follows that $S^{(m)}\in H^p$ only if $S_1^{(m)}\in H^p$.
Moreover, a simple modification of the main result of \cite{MP1982} shows that $S_1^{(m)}\in H^p$ if and only if $p<\frac{1}{2m}$.
Combining these facts, we deduce that $S^{(m)}\in H^p$ (if and) only if $p<\frac{1}{2m}$. This completes the proof.
\end{proof}

\medskip

\noindent
\textbf{Acknowledgements.} The author thanks Toshiyuki Sugawa for valuable comments,
and the referee for careful reading of the manuscript.

\end{document}